\documentclass{amsart}
\usepackage{amssymb,amsthm,amscd}

\newcommand{\fp}{{\mathfrak p}}
\newcommand{\F}{{\mathbb F}}
\newcommand{\Z}{{\mathbb Z}}

\newcommand{\Q}{{\mathbb Q}}

\newcommand{\cO}{{\mathcal O}}

\newcommand{\eps}{\varepsilon}
\newcommand{\fa}{{\mathfrak a}}
\newcommand{\fb}{{\mathfrak b}}

\newcommand{\bmu}{\overline{\mu}}

\newcommand{\Cl}{{\operatorname{Cl}}}

\newtheorem{thm}{Theorem}
\newtheorem{prop}{Proposition}
\newtheorem{lem}{Lemma}

\title{Relations in the $2$-Class Group of Quadratic Number Fields}
\author{F. Lemmermeyer}
\begin{document}

\begin{abstract}
We construct a family of ideals representing ideal classes of
order $2$ in quadratic number fields and show that relations 
between their ideal classes are governed by certain cyclic 
quartic extensions of the rationals.
\end{abstract}

\maketitle

\begin{center} \today \end{center}

\section*{Introduction}

Let $m = p_1 \cdots p_t$ be a product of pairwise distinct primes 
$p_j \equiv 1 \bmod 4$. Then $m$ can be written as a sum of two
squares, say $m = a_j^2 + 4b_j^2$, in $2^{t-1}$ essentially different
ways (i.e., neglecting the signs of $a_j$ and $b_j$). 

We define ideals $\fa_j = (2b_j + \sqrt{m},a_j)$ in the ring of 
integers of the quadratic number field $K = \Q(\sqrt{m}\,)$. Since 
\begin{align*}
  \fa_j^2 & = ((2b_j+\sqrt{m}\,)^2, a_j(2b_j+\sqrt{m}\,), a_j^2) \\
          & = ((2b_j+\sqrt{m}\,)^2, a_j(2b_j+\sqrt{m}\,), m - 4b_j^2) \\
          & = (2b_j+\sqrt{m}\,) (2b_j+\sqrt{m}, a_j, 2b_j-\sqrt{m}\,) \\
          & = (2b_j+\sqrt{m}\,) (4b_j,a_j,2b_j-\sqrt{m}\,) = (2b_j+\sqrt{m}\,)
\end{align*}
is principal, the ideals $\fa_j$ have order dividing $2$.

The "canonical" ideals generating classes of order dividing $2$ are
the products of ramified prime ideals. Letting $\fp_j$ denote the 
prime ideal above the prime $p_j$, we have $\fp_j^2 = (p_j)$. Thus 
each of the $2^t$ ideals
$$ \fb_e = \fp_1^{e_1} \cdots \fp_t^{e_t}, \qquad
                   e = (e_1, \ldots, e_t) \in \F_2^t, $$
generates a class of order dividing $2$. Among these ideal classes
there are two trivial relations coming from the fact that 
$(1) = \prod \fp_j^0$ and $(\sqrt{m}\,) = \prod \fp_j^1$ are principal
ideals (in the usual sense; the ideal class of $(\sqrt{m}\,)$ is principal
in the strict sense if and only if the fundamental unit $\eps$ of $K$ has
norm $-1$).

The following result is well known:

\begin{prop}
There is a nontrivial relation among the ideal classes of the ideals $\fb_e$
if and only of the norm of the fundamental unit $\eps$ of $K$ is $+1$.
\end{prop}

\begin{proof}
If $\prod \fp_j^{e_j} = (\alpha)$ is principal, then 
$(\alpha)^2 = \prod p_j^{e_j}$, hence $\eta = \alpha^2/\prod p_j^{e_j}$ 
is a unit with norm $+1$. If $\eta$ is a square, then so is $\prod p_j^{e_j}$,
which is only possible for $e = (0,\ldots,0)$ and $e = (1,\ldots, 1)$. 
Thus if there is a nontrivial relation among the classes of the ramified
ideals, then there is a nonsquare unit with positive norm; this implies
that $N\eps = +1$.

Conversely, if $N\eps = +1$, then $\eps = \alpha^{1-\sigma}$ for some
$\alpha \in \cO_K$ by Hilbert 90, where $\sigma$ is the nontrivial 
automorphism of $K/\Q$. Since $(\alpha)$ is fixed by the Galois group 
of $K/\Q$, the ideal $(\alpha)$ is a 
product of a rational integer and a ramified ideal. Cancelling the rational 
factors we see that we may assume that $(\alpha)$ is a product of ramified
prime ideals. The equations $\alpha = 1$ and $\alpha = \sqrt{m}$ would imply
$\eps = \pm 1$; thus the relation $(\alpha) = \prod \fp_j^{e_j}$ is necessarily
nontrivial.
\end{proof}

The goal of this article is to clarify the relations between the ideals $\fa_j$ and $\fb_e$.
Our main result is the following

\begin{thm}\label{Th1}
Let $K = \Q(\sqrt{m}\,)$ be a quadratic number field, where $m = p_1 \cdots p_t$
is a product of primes $p_j \equiv 1 \bmod 4$. Let $\eps$ denote the fundamental
unit of $K$.
\begin{enumerate}
\item[a)] If $N \eps = -1$, then the ideal classes $[\fa_j]$ are pairwise distinct and
      represent the $2^{t-1}$ classes of order dividing $2$ in $\Cl(K)$. Each ideal
      $\fa_j$ is equivalent to a unique ramified ideal $\fb_e$. In particular, exactly
      one of the $\fa_j$ is principal; if $\fa_j = (\alpha)$, then 
      $$ \eta = \frac{2b_j + \sqrt{m}}{\alpha^2} $$
      is a unit with norm $-1$ (equal to $\eps$ if $\alpha$ is chosen suitably). 
\item[b)] If $N \eps = + 1$, then there is a subgroup $C$ with index $2$ in the 
      group $\Cl(K)[2]$
      of ideal classes of order dividing $2$ such that each class in $C$ is represented
      by two ramified ideals $\fb_e$ (thus $C$ is the group of strongly ambiguous ideal
      classes in $K$). Each class in $\Cl(K)[2] \setminus C$ is represented
      by two ideals $\fa_j$.
\end{enumerate}
\end{thm}

The proof of Thm. \ref{Th1} uses certain quadratic extensions of 
$\Q(\sqrt{m}\,)$, namely cyclic quartic subextensions of the field 
$\Q(\zeta_m)$ of $m$-th roots of unity. It is perhaps surprising that 
relations in the class group of $K$ are governed by {\em ramified} 
extensions of $K$; note, however, that if $K_1$ and $K_2$ are two 
cyclic quartic extensions as above, then the compositum $K_1K_2$ 
contains a third quadratic extension $K_3/K$, and that this extension 
is unramified: in fact, it is part of the genus class field of $K$.

Here is an example. Let $m = 5 \cdot 13 \cdot 29 = 1885$; then
$m = 6^2 + 43^2 = 11^2 + 42^2 = 21^2 + 38^2 = 27^2 + 34^2$, and we 
consider the ideals 
$\fa_1 = (6 + \sqrt{m},43)$,
$\fa_2 = (42 + \sqrt{m},11)$,
$\fa_3 = (38 + \sqrt{m},21)$,
$\fa_4 = (34 + \sqrt{m},27)$.
Since $N \eps = +1$, none of these ideals is principal.
In fact we have $\fa_2 \sim \fa_3$ and $\fa_1 \sim \fa_4$.
If $\fp$ denotes the ideal with norm $5$, then the ideal classes of 
order $2$ are represented by $\fp$, $\fa_1$ and $\fa_2 \sim \fa_1 \fp$.

The ramified prime ideal above $29$ is principal (in fact,  
$N(87+2\sqrt{1885}\,) = 29$), those above $5$ and $13$ are not. In 
particular, $\sqrt{\eps} = 2 \sqrt{65} + 3 \sqrt{29}$.

\section{Cyclic Quartic Extensions}
Let $m = a^2 + 4b^2$ be a squarefree odd integer; then it is
the discriminant of the quadratic number field $k = \Q(\sqrt{m}\,)$.
In the following, we will always assume that
$m = p_1 \cdots p_t$, where the $p_j \equiv 1 \bmod 4$ are prime 
numbers. 

Some basic facts concerning the description of abelian extensions
of the rationals via characters can be found in \cite{Wash}.
The field of $m$-th roots of unity has Galois group isomorphic
to $(\Z/m\Z)^\times \simeq \prod (\Z/p_j\Z)^\times$, hence has
$G = (\Z/4\Z)^t$ as a quotient. This group $G$ is the Galois
group of the compositum $F$ of the cyclic quartic extensions inside 
the fields $\Q(\zeta_{p_j})$. The character group $X = X(G)$
is generated by quartic Dirichlet characters $\chi_j = \chi_{p_j}$,
and the cyclic quartic subfields of $L$ (and of $\Q(\zeta_m)$) 
correspond to cyclic subgroups of $X$ with order $4$.  

Let $\chi$ be a character generating such a cyclic group of order $4$,
and let $L/\Q$ be the corresponding cyclic quartic extension.
Then $\chi = \prod \chi_j^{e_j}$ with $0 \le e_j < 4$. The order of 
$\chi$ is $4$ if and only if at least one exponent $e_j$ is odd.
By the conductor-discriminant formula, the field $L$ has 
discriminant $m^3$ if and only if all the $e_j$ are odd. Thus these
characters correspond to vectors $e = (e_1, \ldots, e_t)$ with 
$e_j \in \{1, 3\}$. The cyclic quartic extensions $L/\Q$ inside $F$
with discriminant $m^3$ correspond to subgroups generated by such
characters, and since each subgroup contains two such characters
($\chi$ and $\chi^3$), we have

\begin{prop}
Let $m = p_1 \cdots p_t$ be a product of distinct primes $p_j \equiv 1 \bmod 4$.
Then there are $2^{t-1}$ cyclic quartic extensions $L/\Q$ with conductor $m$
and discriminant $m^3$. 
\end{prop}

These extensions can be constructed explicitly:

\begin{prop}\label{Pexp}
Let $m = p_1 \cdots p_t$ be a product of pairwise distinct primes
$p \equiv 1 \bmod 4$. Then there exist $2^{t-1}$ ways of writing
$m = a_j^2 + 4b_j^2$ as a sum of two squares (up to sign).
For each $j$, the extensions
$$ L = \Q(\sqrt{m + 2b_j \sqrt{m}\,}\,) $$
are the $2^{t-1}$ different cyclic quartic extension of $\Q$
with discriminant $m^3$.
\end{prop}

We will prove this result below; now we will use it for giving the
\begin{proof}[Proof of Theorem \ref{Th1}]
Assume first that $N\eps = -1$. We have to show that the classes
of the ideals $\fa_j = (2b_j + \sqrt{m},a_j)$ are pairwise distinct.

Assume to the contrary that $\fa_j \sim \fa_k$; then there exists
some $\xi \in K$ with $\fa_j = \xi \fa_k$. Squaring gives the equation
$(2b_j+\sqrt{m}\,) = \xi^2(2b_k + \sqrt{m}\,)$ of ideals, hence there 
exists a unit $\eta \in \cO_K^\times$ such that
$2b_j+\sqrt{m} = \eta\xi^2(2b_k + \sqrt{m}\,)$. This means
that the square roots of $(2b_j + \sqrt{m})\sqrt{m}$ and 
$\eta(2b_k + \sqrt{m}\,)\sqrt{m}$ must generate the same extension.
Since $\Q(\sqrt{(2b_j + \sqrt{m})\sqrt{m} }\,)$ is a cyclic quartic
extension inside $\Q(\zeta_m)$, so is the extension on the right hand
side. But this implies that $N \eta = +1$, and the fact that $N\eps = -1$
implies that $\eta = \pm \eps^{2n}$. Subsuming the unit into $\xi$ shows 
that we may assume that $\eta = \pm 1$. If $\eta = -1$, the extension on 
the right hand side will ramify at $2$, and this finally shows that $\eta$ 
is a square. Thus we have
$\Q(\sqrt{(2b_j + \sqrt{m})\sqrt{m} }\,) = \Q(\sqrt{(2b_k + \sqrt{m})\sqrt{m} }\,)$
and this can only hold if $j = k$.

We have shown that exactly one ideal $\fa_j$ is principal, say
$\fa_j = (\alpha)$. Since $(\alpha)^2 = \fa_j^2 = (2b_j + \sqrt{m}\,)$
there exists a unit $\eta$ such that $\eta \alpha^2 = 2b_j + \sqrt{m}$.
Taking the norm shows that $N\eta = -1$ as claimed.

If $N \eps = +1$, on the other hand, we first show that none of the ideals
$\fa_j$ is equivalent to a ramified ideal $\fb_e$. In fact, if
$\fa_j = \xi \fb_e$ for some $\xi \in K^\times$, then squaring yields 
$2b_j+\sqrt{m} = \eta \xi^2 m_1$ for $m_1 = \prod p_j^{e_j}$ and some 
unit $\eta$. Taking norms shows that $N\eta = -1$, which contradicts
our assumptions.

Next we show that among the classes of $\fa_j$, each ideal class occurs twice.
In fact, if $\fa_j \sim \fa_k$, say $\fa_j = \xi \fa_k$, then 
$2b_j+\sqrt{m} =  \eta \xi^2 (2b_k+\sqrt{m}\,)$. Up to squares, $\eta$ 
is equal to one of the units $\pm 1, \pm \eps$. As above, $\eta = -1$ 
and $\eta = - \eps$ are impossible, since the places at infinity ramify 
in the extension $K(\sqrt{\eta(2b_k+\sqrt{m}\,)\sqrt{m}}\,)/K$ but not in
$K(\sqrt{(2b_j+\sqrt{m})\sqrt{m}}\,)/K$. Thus either $\eta = 1$ and
$j = k$, or $\eta = \eps$.

Thus each $\fa_j$ is equivalent to at most one other $\fa_k$. Since
there are $2^{t-1}$ ideals $\fa_j$, which are distributed among
the $2^{t-1}$ ideal classes of order $2$ in $\Cl(K)[2]\setminus C$,
it follows that each $\fa_j$ is equivalent to exactly one other $\fa_k$
as claimed.
\end{proof}

\section{Generators of Cyclic Quartic Extensions}

In this section we will give a proof of Prop. \ref{Pexp}.

\subsection*{Kummer generators over $\Q(i)$}
Cyclic quartic extensions $L/\Q$ become Kummer extensions over 
$\Q' = \Q(i)$: with $L' = L(i)$ we have $L' = \Q'(\sqrt[4]{\alpha}\,)$
for some $\alpha \in \Q(i)$. Such an extension $L'$ will be normal 
over $\Q$ if and only if $\alpha^{1-\sigma} = \alpha_\sigma^2$ is a 
square in $\Q'$, where $\sigma$ is the nontrivial automorphism of 
$\Q'/\Q$, and will be abelian over $\Q$ if and only if $\sigma$ acts 
on $\alpha_\sigma$ as on a fourth root of unity, i.e., if and only if 
$\alpha_\sigma^\sigma = \alpha_\sigma^{-1}$.

Since $\alpha_\sigma^{\sigma+1} = 1$, Hilbert 90 implies that 
$\alpha_\sigma = \mu^{\sigma-1} = \bmu/\mu$ for some $\mu \in \Z[i]$. Thus
$\alpha^{1-\sigma} = (\bmu/\mu)^2$. This equation is solved by 
$\alpha = \mu \bmu^3 = m \bmu^2$, where $m = \mu \bmu$ is a positive
integer. Galois theory actually shows that this is the only solution up 
to multiplying $\alpha$ by some nonzero rational number.

We claim that $\alpha$ can be chosen in such a way that $2$ is 
unramified in $L'/\Q'$. Since $k' = \Q'(\sqrt{\alpha}\,) = \Q'(\sqrt{m}\,)$,
the prime above $2$ does not ramify in the quadratic subextension
$k'/\Q'$. In order that $2$ be unramified in $L'/k'$ we have to make 
sure that $\sqrt{\alpha} = \bmu \sqrt{m} \equiv \mu \sqrt{m} \bmod 4$ 
is congruent to a square modulo $4$ in $k'$ (see \cite{Hecke} for the 
decomposition law in quadratic and, more generally, Kummer extensions of
prime degree). If we write $\mu = a + 2bi$, then 
$\mu \equiv \pm 1 + 2i \bmod 4$ if $m \equiv 5 \bmod 8$, hence
$$ \Big( \frac{\pm 1 + 2i+i\sqrt{m}}{1+i}\Big)^2 
                \equiv 2 + (\pm 1 + 2i) \sqrt{m}  
                \equiv (\mp 1 + 2i) \sqrt{m} \bmod 4 $$
since $2 \equiv 2 \sqrt{m} \bmod 4$. 
Similarly, we have $\mu \equiv \pm 1 \bmod 4$ if $m \equiv 1 \bmod 8$, hence     
$$ \Big( \frac{\pm 1+i\sqrt{m}}{1+i}\Big)^2 
                \equiv  \sqrt{m}  \equiv \bmu \sqrt{m} \bmod 4. $$           
Thus $2$ is unramified in $L'/k'$ if we choose $\mu \equiv 1 \bmod 2$.

We have proved

\begin{lem}
If $L/\Q$ is a cyclic quartic extension with conductor $m$, then
there exist integers $a, b$ such that 
$L' = \Q'(\sqrt[4]{\alpha}\,)$ for $\alpha = \mu \bmu^3$, 
where $\mu = a + 2bi \in \Z[i]$ and $m = a^2 + 4b^2$. Replacing 
$a$ by $-a$ or $b$ by $-b$ does not change the extension.
\end{lem}

\subsection*{Kummer Generators over $\Q$}
Let $L' = \Q'(\sqrt[4]{\alpha}\,)$ be a Kummer extension of degree $4$
over $\Q' = \Q(i)$, and assume that $\alpha = \mu \bmu^3$ for some
$\mu = a + 2bi \in \Z[i]$. Set $\beta = \sqrt[4]{\alpha}$; we claim 
that $\beta + \beta'$ is an element of $\Q(\sqrt{m}\,)$, where
$m = \mu \bmu = a^2 + 4b^2$. In fact,
$(\beta + \beta')^2 = \sqrt{m}(\mu + \bmu) + 2m = 2m+2a\sqrt{m}$.
This implies that $L'$ contains a quartic subextension
$L = \Q(\sqrt{2m + 2a\sqrt{m}}\,)$.

The following lemma shows that $L$ is also generated by 
$\sqrt{m + 2b \sqrt{m}}$: 

\begin{lem}
Assume that $Ax^2 - By^2 - Cz^2 = 0$. Then 
\begin{equation}\label{ELeg}
   2(x\sqrt{A} + y\sqrt{B}\,)(x\sqrt{A} + z\sqrt{C}\,)
           = (x\sqrt{A} + y\sqrt{B} + z\sqrt{C}\,)^2.
\end{equation}
\end{lem}

\begin{proof}
We have
\begin{align*}
  (x\sqrt{A}  + y\sqrt{B} + z\sqrt{C}\,)^2 & = 
      Ax^2 + By^2 + Cz^2 + 2xy\sqrt{AB} + 2xz\sqrt{AC} + 2yz\sqrt{BC} \\
    & = 2(Ax^2 + xy\sqrt{AB} + xz\sqrt{AC} + yz\sqrt{BC}\,) \\
    & = 2(x\sqrt{A} + y\sqrt{B}\,)(x\sqrt{A} + z\sqrt{C}\,)
\end{align*}
as claimed.
\end{proof}

Since $m - a^2 - 4b^2 = 0$, the lemma shows that $2(\sqrt{m}+a)(\sqrt{m}+2b)$
is a square in $K = \Q(\sqrt{m}\,)$, hence $2b\sqrt{m} + m$ and 
$2m + 2a \sqrt{m}$ generate the same quadratic extension of $K$.
This finishes the proof of Prop. \ref{Pexp}.

\section*{Acknowledgement}
I thank the referee for carefully reading the manuscript and
for several helpful comments and corrections.

\end{document}